\documentclass[10pt]{amsart}
\usepackage{caption}
\usepackage{subcaption}
\usepackage{float}
\usepackage{comment}
\usepackage{enumerate,amssymb,  mathrsfs, graphicx}
\newtheorem{theorem}{Theorem}[section]
\newtheorem{lemma}[theorem]{Lemma}
\newtheorem*{lemma*}{Lemma}

\newtheorem{proposition}[theorem]{Proposition}
\newtheorem{corollary}[theorem]{Corollary}
\theoremstyle{definition}
\newtheorem{definition}[theorem]{Definition}
\newtheorem{example}[theorem]{Example}

\newtheorem{remark}[theorem]{Remark}

\theoremstyle{remark}

\numberwithin{equation}{section}

\begin{document}
\title[Conformal and holomorphic barycenter]{Conformal and holomorphic barycenters in hyperbolic balls}

\author[V. Ja\'cimovi\'c]{Vladimir Ja\'cimovi\'c}
\address{University of Montenegro, Faculty of natural sciences and mathematics, Podgorica, Cetinjski put b.b. 81000 Podgorica, Montenegro }
\email{vladimirj@ucg.ac.me }

\author[D. Kalaj]{David Kalaj}
\address{University of Montenegro, Faculty of natural sciences and mathematics, Podgorica, Cetinjski put b.b. 81000 Podgorica, Montenegro }
\email{davidk@ucg.ac.me}

\footnote{2020 \emph{Mathematics Subject Classification}: Primary 30J10}


\keywords{M\"obius tranformations, automorphisms, unit ball, convex potentials}


\maketitle

\begin{abstract}
We introduce the notions of \textit{conformal barycenter} and \textit{holomorphic barycenter} of a measurable set $D$ in the hyperbolic ball. The two barycenters coincide in the disk, but they differ in multidimensional balls $\mathbb{C}^m \cong \mathbb{R}^{2m}$. These notions are counterparts of barycenters of measures on spheres, introduced by Douady and Earle in 1986.
\end{abstract}

\section{Introduction} \label{Intro}

Barycenter is the "center of mass" of a collection of points, weighted by their respective masses. For a set of points $\{x_i\}$ in the Euclidean space with corresponding weights $\{w_i\}$, the barycenter is their weighted average.

A conformal barycenter extends this notion to settings where we are dealing with distances and structures which are preserved under conformal maps. It is often determined as the point that minimizes a certain energy functional or as a fixed point of some iterative conformal process.

We start with the notion of barycenter in the unit disc $\mathbb{B}^2 = \{z \in \mathbb{C} \; : \; |z| \leq 1\}$ in the complex plane. Denote by $G$ the group of conformal automorphisms of $\mathbb{B}^2$ and by $G_+$ the subgroup of orientation preserving maps. The group $G_+$ consists of transformations of the following form
\begin{equation}
\label{Mobius_disc}
g_a(z) = e^{i \theta} \frac{a-z}{1-\bar a z}, \quad \theta \in [0,2 \pi), \; a \in \mathbb{B}^2.
\end{equation}
Denote by $d\lambda(z)$ the Lebesgue measure in the complex plane, then the hyperbolic measure reads
\begin{equation}
\label{hyp_meas_disc}
d \Lambda(z) = \frac{d \lambda(z)}{(1-|z|^2)^2}.
\end{equation}
The following assertions hold.

\begin{theorem} \label{disc2}
Let $D \subseteq \mathbb{B}^2$ be a Lebesgue-measurable set, such that $0<\Lambda(D)< + \infty$. Then there exists a unique  point  $c=c(D)\in \mathbb{B}^2$ such that  $$\int_{D} g_c(z) \mathrm{d}\Lambda(z)=0$$
where $g_c$ is the M\" obius transformation of the unit disc defined by (\ref{Mobius_disc}) with arbitrary $\theta \in [0,2 \pi)$.
\end{theorem}
\begin{definition}
We say that the point $c$ is the conformal barycenter of the set $D$.
\end{definition}

For a Lebesgue-measurable set $A \subseteq \mathbb{B}^2$ consider the following function
\begin{equation}
\label{potential_disc1}
H(z) = - \int_A \log \frac{(1-|z|^2)(1-|\zeta|^2)}{|1 - z \bar \zeta|^2} d \Lambda(\zeta).
\end{equation}
Then we prove
\begin{theorem}
For any Lebesgue-measurable set $A$, such that $0<\Lambda(A)<+ \infty$ the following assertions hold:

\begin{enumerate}
\item
The function $H(z)$ has a unique global minimum on $\mathbb{B}^2$.
\item
The minimum of $H(z)$ is the conformal barycenter of $A$.
\end{enumerate}
\end{theorem}
Moreover
\begin{theorem}
Let $\zeta_1,\dots, \zeta_N$ be points in the unit disk $\mathbb{B}^2$. Then there exists unique (up to a rotation) M\"obius transformation of the form \eqref{Mobius_disc}, such that $$\sum_{k=1}^N g_c(\zeta_k)=0.$$
\end{theorem}

\begin{definition}
The point $c$ from the above Theorem is said to be the conformal barycenter of points $\zeta_1,\dots, \zeta_N$.
 \end{definition}

For given points $\zeta_1,\dots,\zeta_N$ consider the following function
$$
H_N(z) = - \sum \limits_{i=1}^N \log \frac{(1-|z|^2)(1-|\zeta_i|^2)}{|1-\bar z \zeta_i|^2}.
$$
Then
\begin{theorem}
\begin{enumerate}
\item
The function $H_N(z)$ has a unique global minimum on $\mathbb{B}^2$.
\item
The minimum of $H_N(z)$ is the conformal barycenter of points $\zeta_1,\dots,\zeta_N$.
\end{enumerate}
\end{theorem}
Furthermore
\begin{theorem} \label{Conf_invar}
 The conformal barycenter is conformally invariant. In other words, if $c$ is conformal barycenter of a set $D$, then $g(c)$ is conformal barycenter of $g(D)$ for any M\"obius transformation $g \in G_+$.
 \end{theorem}

In the present paper we will prove generalizations of all the above statements for the case of unit balls and extend the definition of conformal barycenter to higher dimensions. In fact, for even-dimensional balls we will present two extensions which correspond to two non-equivalent metrics in balls.

\begin{remark}

The notion of conformal barycenter was first introduced by Douady and Earle in their seminal paper \cite{actaex}. With each probability measure $\mu$ on the unit circle $\mathbb{S}^1$ they associated a vector field in $\mathbb{B}^2$ in the following way  \begin{equation}\label{xi}\xi_\mu(w)=(1-|w|^2)\int_{\mathbb{S}^1}\frac{\zeta-w}{1-\zeta\bar{w}}d\mu(\zeta).\end{equation}
It is proven in \cite{actaex} that for each probability measure $\mu$ which does not contain heavy atoms there is a unique point $b(\mu)$ in $\mathbb{B}^2$ at which vector field $\xi_\mu(\cdot)$ vanishes. This point is said to be {\it conformal barycenter} of the measure $\mu$ and it is conformally invariant in the sense of Theorem \ref{Conf_invar}.

Authors further exploited properties of the conformal barycenters in order to demonstrate that any quasi-symmetric homeomorphism of the circle can be extended to the homeomorphism of disk in a conformally natural way.
The definition (and the uniqueness property) of the conformal barycenter of a probability measure on the circle extends to higher dimensions (e.g. to probability measures on spheres), see \cite{actaex,cas}. \footnote{However, notice that the Douady-Earle conformally natural extension of homeomorphisms on spheres are not necessarily homeomorphisms in balls.}
In the present paper we build upon the idea of Douady and Earle and introduce conformal barycenters of subsets (or measures) of hyperbolic balls.
\end{remark}

\section{Preliminaries}

In order to facilitate the exposition and to avoid confusion with notations, we have to start with some common facts and concepts of Riemannian geometry.

\subsection{Preliminaries from the Riemannian geometry}

 A \textit{Riemannian metric} on a manifold \(M\) is a smoothly varying, positive-definite, symmetric bilinear form \(g_p\) on the tangent space \(T_pM\) at each point \(p \in M\).

More formally, for each point \(p \in M\), the metric \(g_p\) is a function:
\[
g_p : T_pM \times T_pM \rightarrow \mathbb{R}
\]
which satisfies:
\begin{itemize}
    \item \textit{Symmetry}: \( g_p(v, w) = g_p(w, v) \) for all \( v, w \in T_pM \),
    \item \textit{Bilinearity}: \( g_p(av + bw, u) = a g_p(v, u) + b g_p(w, u) \) for all \(a, b \in \mathbb{R}\) and \( v, w, u \in T_pM \),
    \item \textit{Positive-definiteness}: \( g_p(v, v) > 0 \) for all non-zero \(v \in T_pM\).
\end{itemize}

The metric provides a way of measuring:
\begin{itemize}
    \item \textit{Length of a vector} \( v \in T_pM \): \( |v| = \sqrt{g_p(v, v)} \),
    \item \textit{Angle between two vectors} \( v, w \in T_pM \):
    \[
    \cos(\theta) = \frac{g_p(v, w)}{|v||w|}
    \].
\end{itemize}

A \textit{Riemannian manifold} \( (M, g) \) is a smooth manifold \(M\) equipped with a Riemannian metric \(g\).

\textit{Length of a Curve.}
Given a smooth curve \( \gamma : [a, b] \to M \) on a Riemannian manifold \( (M, g) \), the \textit{length of the curve} is defined as:
\[
L(\gamma) = \int_a^b \sqrt{g_{\gamma(t)}\left(\dot{\gamma}(t), \dot{\gamma}(t)\right)} \, \mathrm{d}t,
\]
where \( \dot{\gamma}(t) \) is the velocity vector (tangent vector) of the curve at time \(t\), and \( g_{\gamma(t)}\left(\dot{\gamma}(t), \dot{\gamma}(t)\right) \) is the square of the speed (i.e. the norm squared of the velocity vector, measured in the Riemannian metric).

\textit{Distance.}
The \textbf{distance} between two points \(p, q \in M\) on a Riemannian manifold is defined as the infimum of the lengths of all smooth curves connecting \(p\) and \(q\). Formally,
\[
d_M(p, q) = \inf \{ L(\gamma) \mid \gamma : [a, b] \to M, \, \gamma(a) = p, \gamma(b) = q \}.
\]

\textit{Geodesics.}
A \textit{geodesic} is a curve that locally minimizes distance between points. In a Riemannian manifold, geodesics generalize the concept of "straight lines" in Euclidean space. Formally, geodesic \( \gamma(t) \) is a curve that satisfies the \textit{geodesic equation}, which is the second-order differential equation derived from the Riemannian metric.

\textit{Sectional Curvature.}
The sectional curvatures of a Riemannian manifold \( (M, g) \) is a measure of how the manifold curves in different directions at a given point. It is associated with 2-dimensional planes inside the tangent space of the manifold at each point.

Let \( (M, g) \) be a Riemannian manifold, and let \( p \in M \) be a point on the manifold. Consider a 2-dimensional plane \( \sigma \subset T_pM \) inside the tangent space at \(p\), spanned by two linearly independent tangent vectors \( v, w \in T_pM \). The sectional curvature \( K(\sigma) \) of the plane \( \sigma \) is a quantity defined in terms of the \textit{Riemann curvature tensor} \( R \) as follows:
\[
K(\sigma) = \frac{g_p\left( R(v, w)w, v \right)}{g_p(v, v) g_p(w, w) - g_p(v, w)^2}.
\]
This quantity measures how the manifold curves along the 2-dimensional subspace spanned by \( v \) and \( w \).

Here \( R(v, w) \) is the action of the Riemann curvature tensor on the vectors \( v \) and \( w \), while the denominator \( g_p(v, v) g_p(w, w) - g_p(v, w)^2 \) is the area of the parallelogram formed by the vectors \( v \) and \( w \) in the tangent space.

Two particular examples of Riemann manifolds are important for this paper. Before that, let us introduce some notation.
By $\mathbb{B}^n$ and $\mathbb{B}_m$ we denote the unit balls in $\mathbb{R}^n$ and in $\mathbb{C}^m$, respectively. We will sometimes use the notation $\mathbb{B}$ to denote any ball.

The norm of a vector $x=(x_1,\dots,x_n)\in \mathbb{R}^n$ is denoted by $|x|=\sqrt{\sum_{k=1}^n x_k^2}$. The norm on of a vector $z=(z_1,\dots,z_m) \in \mathbb{C}^m$ is $|z|=\sqrt{\left<z,z\right>}$, where $\left<z,w\right>=\sum_{k=1}^m z_k\overline{w_k}.$

By $\mathrm{d}\lambda(x)$ and $\mathrm{d}\lambda(z)$ we denote the Lebesgue measure in $\mathbb{R}^n$ and in $\mathbb{C}^m$ respectively. Then the hyperbolic measures read
\begin{equation} \label{hyp_meas}
\mathrm{d}\Lambda(x) =\frac{\mathrm{d}\lambda(x)}{(1-|x|^2)^n},\ \  \ \mathrm{d}\Lambda(z) =\frac{\mathrm{d}\lambda(z)}{(1-|z|^2)^{n+1}}.
\end{equation}
 We say that the set $D\subset\mathbb{B}$ is measurable if it is Lebesque-measurable and if $\Lambda(D):=\int_{D}\mathrm{d}\Lambda<\infty$.
\begin{example}
Let $\mathbb{B}^n$ be the unit ball in $\mathbb{R}^n$ equipped with the metric $$g_x(u,v) =  \frac{\left<u,v\right>}{(1-|x|^2)}, u, v\in \mathbb{R}^n.$$ We call $(\mathbb{B}^n, g)$ the hyperbolic ball. It is well-known that the hyperbolic ball has constant negative sectional curvature.
\end{example}

\begin{example}
Let $\mathbb{B}_m$ be the unit ball in $\mathbb{C}^m$ equipped with the metric $$g_z(u,v) =\Re \left<B(z)u,v\right>, \ \ u, v\in \mathbb{C}^m, z\in \mathbb{B}_m.$$ Here $$B(z)=(b(z)_{ij})_{i,j=1}^n \quad \mbox{ and } \quad b(z)_{ij}= \frac{1}{n+1}\frac{\partial^2}{\partial \overline{{z_i}}\partial z_j}K(z,z),$$ where $$K(z,w)=\frac{1}{n+1}\frac{1}{(1-\left<z,w\right>)^{n+1}}$$ is the Bergman kernel \cite{kezu}.

The Riemannian manifold $(\mathbb{B}_m, g)$ is named the Bergman  ball. Bergman balls have constant negative sectional curvature (\cite{arxiv}).
\end{example}

\subsubsection{Poincar\'e distance and M\"obius transformations of the unit ball}\label{sub1}

The Poincar\'e distance is given by
\begin{equation}\label{pome}d_h(x,y)=\frac{1}{2}\log \frac{1+R}{1-R},\end{equation} where $$ R=\frac{|x-y|}{\sqrt{\rho(x,y)}} \mbox{ and } \rho(x,a)=|x-a|^2+(1-|a|^2)(1-|x|^2).$$   M\"obius transformations of the unit ball, up to orthogonal transformation of the Euclidean space  are given by
\begin{equation} \label{Mobius_ball}
y=h_a(x)=\frac{a|x-a|^2+(1-|a|^2)(a-x)}{\rho(x,a)}.
\end{equation}
  It is well-known that the Poincar\'e metric is invariant under the action of  M\"obius transformations of the unit ball onto itself. Moreover $h_c^{-1}(x)=h_c(x)$ for every $c\in \mathbb{B}$.

 Now if $c\in \mathbb{B}$ is arbitrary and  $m$ is any M\"obius transformation preserving the unit ball, then there exists a orthogonal transformation $A$, such that  \begin{equation}\label{arbmo}(h_{m(c)}\circ m)(x) = (A \circ h_c) (x).\end{equation} In order to verify this, observe that $$h_{m(c)}( m(c))=0$$ and
$$|h_{m(c)}(m(0))|=\frac{|m(c)-m(0)|}{\sqrt{\rho(m(c), m(0))}}=\frac{|c-0|}{\sqrt{\rho(c,0)}}=|c|.$$
Let $A=h_{m(c)}\circ m \circ h_c$. Since $A(0)=0$ and $|A(c)|=|c|$, we infer that $A$ is an orthogonal transformation. This implies \eqref{arbmo}.

We conclude this subsection with several formulae that will be useful in the sequel, see \cite{alfors}

\begin{equation}\label{poMT}d_h(x,y)=\frac{1}{2} \log \frac{1+|h_a(x)|}{1-|h_a(x)|};\end{equation}

\begin{equation}\label{jaka}(1-|h_a(x)|^2)=\frac{(1-|a|^2)(1-|x|^2)}{\rho(x,a)}.\end{equation}

Finally, Jacobian of the mapping $y=h_a(x)$ is given by  \begin{equation}\label{jaco}J(y,x)=\frac{1-|a|^2}{\rho(a,x)^n}=\frac{(1-|y|^2)^n}{(1-|x|^2)^n}.\end{equation}



\subsubsection{Bergman distance and automorphisms of the unit ball $B\subset \mathbb{C}^m$}\label{sub2}

Let $P_a$ be the
orthogonal projection of $\mathbb{C}^n$ onto the subspace $[a]$ generated by $a$, and let $$Q=Q_a =
I - P_a$$ be the projection onto the orthogonal complement of $[a]$. Explicitly, $P_0 = 0$ and  $P=P_a(z) =\frac{\left<z,a\right> a}{\left<a, a\right>}$. Set $s_a = (1 - |a|^2)^{1/2}$ and consider the map
\begin{equation}
\label{Bergman_transf}
p_a(z) =\frac{a-P_a z-s_a Q_a z}{1-\left<z,a\right>}.
\end{equation}
Compositions of mappings of the form \eqref{Bergman_transf} and unitary linear mappings of the $\mathbb{C}^n$ consitute the group of holomorphic automorphisms of the unit ball $\mathbb{B}_n\subset \mathbb{C}^m$. It is easy to verify that $p_a^{-1}=p_a$. Moreover, for any automorphism $q$ of the Bergman ball onto itself there exists a unitary transformation $U$ such that \begin{equation}\label{automob}p_{q(c)}\circ q=U\circ p_c.\end{equation}

By using the representation formula \cite[Proposition~1.21]{kezu}, we can introduce the Bergman metric as \begin{equation}\label{bergmet}d_B(z,w)=\frac{1}{2}\log\frac{1+|p_w(z)|}{1-|p_w(z)|}.\end{equation}
If $\Omega= \{z\in \mathbb{C}^n:\left<z,a\right>\neq 1\}$,  then the map $p_a$  is holomorphic in $\Omega$. It is clear that $\overline{\mathbb{B}_n}\subset \Omega$ for $|a| < 1.$

It is well-known that every automorphism $q$ of the unit ball is an isometry w.r. to the Bergman metric, that is: $d_B(z,w)=d_B(q(z),q(w))$.

We also point out the formulae \begin{equation}\label{phia}(1-|p_a(z)|^2)=\frac{(1-|z|^2)(1-|a|^2)}{|1-\left<a,z\right>|^2}\end{equation} and the expression for the Jacobian $$J(z,p_a)=\left(\frac{1-|p_a(z)|^2}{1-|z|^2}\right)^{n+1}=\left(\frac{1-|a|^2}{|1-\left<z,a\right>|^2}\right)^{n+1}$$
which will be needed in the sequel.


For all the above facts we refer to monographs by Zhu \cite{kezu} and Rudin \cite{rudin}.

\section{Potentials in hyperbolic balls}

Throughout the paper we will consider those measures $\mu$ defined on the  Borel sigma-algebra on $\mathbb{B}\subset \mathbb{R}^n (\mathbb{C}^m)$ which satisfy the following Lusin's condition (N):

(N) \quad If $\mu(D)=0$, then $\mu(g(D))=0$ (respectively $\mu(h(D))=0$) for any Borel measurable set $D \subseteq \mathbb{B}$ and any M\" obius transformation $g$ of the unit ball $\mathbb{B}^n$ (respectively any holomorphic automorphism $h$ of the unit ball $\mathbb{B}_m$).

 In the present Section we prove the following
\begin{theorem}\label{prop2}
Let $\mu$ be a measure on the unit ball $\mathbb{B}$ which satisfies Lusin's condition $\mathrm{(N)}$ and $A \subseteq \mathbb{B}$ a $\mu$-measureable set, such that $0< \mu(A) < + \infty$.
\begin{enumerate}
\item
 The function  \begin{equation}\label{hane}
G(x) = -\int_{A}\log \frac{(1-|x|^2)(1-|y|^2)}{|x-y|^2+(1-|y|^2)(1-|x|^2)}\mathrm{d}\mu(y), \quad x \in \mathbb{B}^n
\end{equation}
has a unique minimum in $\mathbb{B}$.
\item
The function
\begin{equation}\label{marte}
L(z) = -\int_{A}\log \frac{(1-|z|^2)(1-|w|^2)}{|1-\left<z,w\right>|^2}\mathrm{d}\mu(w), \quad z \in \mathbb{B}_m
\end{equation}
has a unique minimum in $\mathbb{B}$.
\end{enumerate}
\end{theorem}

\subsection{Auxiliary results}

In order to prove the above theorem we need some results about geodesic convexity of distance functions in hyperbolic balls.

\begin{definition}
We say that $f:M\to \mathbb{R}$ is geodesically (strictly convex), if for every pair $a,b$ of different points a geodesic line $\gamma:[0,1]\to M$ that connects $a$ and $b$ with $\gamma(0)=a$ and $\gamma(1)=b$, we have that $f(\gamma(t))< (1-t) f(a)+t f(b)$ ($f(\gamma(t))\le  (1-t) f(a)+t f(b)$), $0< t< 1$. Note that, the parametrization $\gamma$ of geodesic line $\gamma[0,1]$ is is proportionally to arc length.
\end{definition}\label{propo32}
\begin{proposition}\label{prop} Geodesically strictly convex function has no more than one local minimum in the unit ball.
\end{proposition}

\begin{proof}
Suppose that $x$ is a local minimum of $f$.
Then there is $\epsilon>0$ so that $f(y) \ge  f(x)$ for $x\in B(x, \epsilon)$.

Suppose that there is $z\in \mathbb{B}$  with
$$f(z) < f(x).$$
Convexity of $\mathbb{B}$ implies
that geodesic $\gamma(t)$  which connects $z$ and $x$ is contained in $\mathbb{B}$ for $t \in  [0, 1]$.
By convexity of $f$, we have
$$f(\gamma(t)) = t f(x) + (1 - t)f(z)< t f(x) + (1 - t)f(x) = f(x).$$
But, as $t \to 1$, $\gamma(t)\to x$ and the previous inequality contradicts local optimality
of $x$.
\end{proof}

\begin{definition}\cite[Definition~3.3.5]{jost}
The Hessian of a differentiable function $f: M\to R$ on a Riemannian manifold $M$ is $\nabla \mathrm{d} f.$

We have $\mathrm{d}f=\sum_{i=1}^n\frac{df }{\mathrm{d}x^i} \mathrm{d}x^i$ in local coordinates, hence $$\nabla df=\left(\frac{\partial^2 f}{\partial x^i \partial x^j}-\sum_{k=1}^n\frac{\partial f}{\partial x^k}\Gamma^k_{ij}\right)\mathrm{d}x^i\otimes \mathrm{d}x^j.$$
\end{definition}
where $\Gamma^k_{ij}$ are  Christoffel symbols.
\begin{proposition}
$f$ is strictly geodesically  convex if its Hessian $\nabla \mathrm{d}f$ is positive definite.
\end{proposition}
\begin{corollary}\label{corr}
If $a$ is a stationary point of geodesically strictly convex function $f$, then $a$ is global minimum of $f$.
\end{corollary}
\begin{proof}
Since $a$ is stationary point, we have that $\frac{\partial f}{\partial x^k}(a)=0$ for every $k=1,\dots,n$. Hence, the matrix $\left(\frac{\partial^2 f}{\partial x^i \partial x^j}(a)\right)_{{i,j}=1}^n$ is positive definite which implies that $a$ is a local minimum of $f$. By Proposition~\ref{prop}, we conclude that $a$ is a global minimum of $f$.
\end{proof}

Now the following lemma is the main step of proof of main results.
\begin{lemma}\label{l1}
Functions $f(x) = d_h(x,p)$ and $F(z)=d_B(z,q)$ are  geodesically  convex.
\end{lemma}
\begin{proof}[Proof of Lemma~\ref{l1}]

Let us prove the assertion for $d_h$ and notice that the same proof applies to $d_B$. We will use the fact that metrics $d_h$ and $d_B$ have negative sectional curvature.

Let $\gamma:[0,1]\to D$ be the geodesic line connecting $a$ and $b$ so that $\gamma(t)$ divides the geodesic arc $ab$ into the ratio $t: 1-t$.
We need to prove that
\begin{equation}
\label{w1}
d_h(p, \gamma(t))\le (1-t)d_h(p, \gamma(0))+t d_h(p, \gamma(1)).
\end{equation}

We start from the triangle inequality
$$|d_h(p, \gamma(0))-d_h(p, \gamma(1))|\le d_h(\gamma(0),\gamma(1)).$$
This inequality is equivalent with
$$d_h^2(p, \gamma(1))+d_h^2(p, \gamma(0))-d_h^2(\gamma(0), \gamma(1))\le 2 d_h(p, \gamma(1))d_h(p, \gamma(0))$$
which is in its turn equivalent with
\begin{equation}
\label{w3}
(1-t)d_h^2(p, \gamma(0))+t d_h^2(p, \gamma(1))- t (1-t)d_h^2(\gamma(0), \gamma(1))$$ $$\le ((1-t)d_h^2(p, \gamma(0))+t d_h^2(p, \gamma(1)))^2.
\end{equation}

On the other hand, the following formula holds (see \cite[eq. 4.8.7]{jost})
\begin{equation}
\label{w2}
d_h^2(p, \gamma(t))\le (1-t)d_h^2(p, \gamma(0))+t d_h^2(p, \gamma(1))- t (1-t)d_h^2(\gamma(0), \gamma(1)).
\end{equation}

By comparing inequalities (\ref{w2}) and (\ref{w3}) we obtain (\ref{w1}) which completes the proof.

\end{proof}
\subsection{Proof of Theorem~\ref{prop2}}

\begin{proof}

Using relations \eqref{poMT} and \eqref{jaka} we have that
 $$-\log \frac{(1-|x|^2)(1-|y|^2)}{\rho(y,x)} = \log \cosh^2 \left( d_h(x,y) \right).$$

Hence, function (\ref{hane}) can be written as
$$
G(x) = \int_{A} \log \cosh^2 \left( d_h(x,y) \right)\mathrm{d}\mu(y).
$$
We have already checked that $d_h(x,w)$ is a convex function of $x$ for fixed $w$. Now, since $\log \cosh^2 t$ is an increasing convex function of the real variable $t$ and its second derivative equals to $ 2 \mathrm{sech}^2 t$, its integral is convex. Hence, the function $G$ is strictly convex.

Let $$P(x):= -\log \frac{(1-|x|^2)(1-|y|^2)}{\rho(y,x)}.$$ Since  $\lim_{|x|\to 1}P(x)=+\infty$, it follows that $\lim_{|x|\to 1}G(x)=+\infty$. Since  $P$ is smooth, so is $G$ and   $G$ has the unique minimum for $a\in \mathbb{B}$ (because of Proposition~\ref{propo32}).

The second point of Theorem \ref{prop2} for the function (\ref{marte}) can be proven in an analogous way by using Lemma~\ref{l1} and relations \eqref{bergmet} and \eqref{phia}.

\end{proof}

\begin{remark}
Although Theorem \ref{prop2} is valid for any measure $\mu$, it is particularly meaningful in those special cases when the measure $\mu$ is conformally (or holomorphically) invariant (meaning that $\mu(D)=\mu(g(D))$ for any subset $D\subset \mathbb{B}$ and any automorphism $g$). These special cases of Theorem \ref{prop2} are emphasized throughout our further exposition.

\end{remark}

\section{Barycenters in Poincar\' e balls} \label{balls}

We first introduce the notion of barycenter of a set w.r. to any measure $\mu$.

\begin{theorem}\label{mainth1}

Let $\mu$ be a measure in the unit ball $\mathbb{B}^n \subset \mathbb{R}^n$ which satisfies Lusin's condition $\mathrm{(N)}$ and $D$ a $\mu$-measurable subset of $\mathbb{B}^n$, such that $0 < \mu(D) < + \infty$.
\begin{enumerate}
\item
There is a unique  point  $b=b(D)\in \mathbb{B}$, such that  $$\int_{D}h_b(x) \mathrm{d}\mu(x)=0$$
where $h_b$ is M\" obius transformation given by \eqref{Mobius_ball}.
\item
Point $b(D)$ is unique minimum of the function \eqref{hane}.
\end{enumerate}
\end{theorem}

\begin{proof}
Consider the function $$G(x) = -\int_{A}\log \frac{(1-|x|^2)(1-|y|^2)}{\rho(x,y)}\mathrm{d}\mu(y).$$

By Theorem \ref{prop2}, $G$ has a unique minimum $a\in \mathbb{B}$. Let $h_a$ be a M\"obius transformation of the unit ball onto itself so that $h_a(0)=a$ and $h_a\circ h_a=\mathrm{id}$. Then the function $G_1(x) = g(h_a(x))$ has unique minimum at $x=0$.

Moreover
\[
\begin{split}
G_1(x) &=  -\int_{A}\log \frac{(1-|h_a(x)|^2)(1-|y|^2)}{\rho(h_a(x),y)}\mathrm{d}\mu(y)
\\&= \int_{A}\log \cosh^2 \left( {d(h_a(x),y)} \right)\mathrm{d}\mu(y)
\\&= \int_{A}\log \cosh^2 \left( {d(x,h_a^{-1}(y))} \right)\mathrm{d}\mu(y)
\\&= \int_{A}\log \cosh^2 \left( {d(x,h_a(y))} \right)\mathrm{d}\mu(y).
\end{split}
\]
Then $$\nabla G_1(x)=\int_{A}\left(\frac{2 x}{1-|x|^2}+ \frac{2x|h_a(y)|^2-2 h_a(y)}{\rho(x, h_a(y))}\right)\mathrm{d}\mu(y) .$$

To justify the differentiation under the integral, observe that $$\left|\frac{2 x}{1-|x|^2}+ \frac{2x|h_a(y)|^2-2 h_a(y)}{\rho(x, h_a(y))}\right|\le \frac{2|x|}{1-|x|^2}+\frac{2+2|x|}{(1-|x|^2)^2}$$ and recall that we require $\mu(A)<\infty$.

Hence,
$$\nabla G_1(0)=-2\int_{A} h_a(y)\mathrm{d}\mu(y).$$ Since $x=0$ is the stationary point of $G_1$, it follows that $$\int_{A} h_a(y)\mathrm{d}\mu(y)=0,$$ which completes the proof.

\end{proof}

\begin{definition}
We say that point $b(D)$ from the above Theorem is barycenter of the set $D$ w.r. to measure $\mu$.
\end{definition}

\subsection{Conformal barycenter in the Poincar\' e ball}

\begin{definition}
Barycenter of a set $D \in \mathbb{B}^n$ w.r. to the hyperbolic measure $\Lambda(x)$ defined in (\ref{hyp_meas}) is named {\it conformal barycenter}.
\end{definition}
We will denote the conformal barycenter of $D$ by $c \equiv c(D)$.

From Theorem \ref{mainth1}(1) it follows that the conformal barycenter of $A$ is minimum of the function where the measure $\mu(y)$ is replaced by $\Lambda(y).$

The most transparent and potentially important for applications case is when the set $D$ is finite. In order to address this case, we apply Theorem \ref{mainth1} with $\mu$ being the counting measure:
\begin{equation}
\label{count}
\mu(A) = \left\{ \begin{array}{ll}
         |A| & \mbox{if $A$ is finite};\\
        \infty & \mbox{if $A$ is infinite}.\end{array} \right.
        \end{equation}

Such a choice of $\mu$ yields the following

\begin{corollary}\label{t1}
Assume that $x_1,\dots, x_N$ are points on the unit ball $\mathbb{B}^n\subset \mathbb{R}^n$.
\begin{enumerate}

\item
 There exists a unique (up to a linear isometry) M\"obius transformation $h$ of the unit ball onto itself, such that $$\sum_{k=1}^N h(x_k)=0.$$

\item
Decompose the M\"obius transformation $h$ as $h=A\circ h_c$ for some $c\in \mathbb{B}^n$ and a linear isometry $A$ of the unit ball.
Then point $c$ is unique minimum of the function
$$
G_N(y) = - \sum \limits_{i=1}^N \log \frac{(1-|y|^2)(1-|x_i|^2)}{|y-x_i|^2+ (1-|x_i|^2)(1-|y|^2)}, \quad y \in \mathbb{B}^n.
$$
\end{enumerate}
\end{corollary}

\begin{definition} \label{t3}
The point $c$ from Corollary \ref{t1}(2) is said to be conformal barycenter of the set $\{ x_1,\dots,x_N \}$
\end{definition}

\begin{theorem}
Conformal barycenter is conformally invariant. In other words, if $c=c(D)$ is conformal barycenter of $D$, then $h(c)$ is conformal barycenter of $h(D)$ for any   M\"obius transformation $h$ of the unit ball.
\end{theorem}

\begin{proof}
We aim to prove that if $$\int_{D} h_c(x) \mathrm{d}\Lambda(x)=0,$$ and $q$ is any  M\"obius transformation, then $$\int_{q(D)} h_{q(c)}(y) \mathrm{d}\Lambda(y)=0.$$
Introduce the change of variables $y=g(x)$. Then by \eqref{jaco}, we have $\mathrm{d}\Lambda(y)=\mathrm{d}\Lambda(x)$. Thus

$$\int_{q(D)} h_{q(c)}(y)\mathrm{d}\Lambda(y)=\int_{D}h_{q(c)}(q(x))\mathrm{d}\Lambda(x).$$ Now, taking into account considerations in Subsection~\ref{sub1}, we obtain that $$h_{q(c)} (q(x))=A h_c(x),$$ for an orthogonal transformation $A$ of the unit ball onto itself. Therefore $$\int_{q(D)} h_{q(c)}(y)\mathrm{d}\Lambda(y)=A\int_{D} h_c\mathrm{d}\Lambda(x)=0$$  which confirms conformal invariance.

The above proof can easily be adapted to demonstrate conformal invariance of the barycenter of a finite set in the sense of Definition \ref{t3}.
\end{proof}

\section{Barycenters in Bergman balls} \label{Bergman}

\begin{theorem}\label{mainth2}
Let $\mu$ be a measure in $\mathbb{B}_m \subset \mathbb{C}^m$ which satisfies Lusin's condition (N) and $K$ a $\mu$-measurable subset of $\mathbb{B}_m$, such that $0 < \mu(K) < + \infty$.
\begin{enumerate}
\item
There exists a unique point $a \equiv a(K) \in \mathbb{B}_m$, such that  $$\int_{K}p_a(z) \mathrm{d}\mu(z)=0,$$
where $p_a$ is the map defined by (\ref{Bergman_transf}).
\item
Point $a(K)$ is the minimum of the function \ref{marte}.
\end{enumerate}
\end{theorem}

\begin{proof}
As before, by using \eqref{bergmet} we obtain again
$$L(z) = \int_{A}\log \cosh^2 \left( {d_B(z,w)} \right)\mathrm{d}\mu(w).$$

The function $L$ has the unique minimum $a$ in $\mathbb{B}_m$. Let $p_a$ be an involutive automorphism of the the unit ball onto itself so that $p_a(0)=a$. Then zero is the unique minimum of $L_1(z) = g(p_a(z))$. Since automorphisms are isometries in the Bergman metric $d_B$, we obtain  $$L_1(z) = \int_{A}\log \cosh^2 \left( {d_B(z,p_a(w))} \right)\mathrm{d}\mu(w).$$
Moreover its gradient $$(\partial_{x_1}+i\partial_{y_1}, \dots, \partial_{x_n}+ i \partial_{y_n})$$ in $\mathbb{C}^m\cong \mathbb{R}^{n},$ $m=2n$ is given by

$$\nabla L_1(z)= \int_{A} \left( -\frac{2 p_a(w)}{1-\left<z,p_a(w)\right>}+\frac{2 z}{1-|z|^2} \right) \mathrm{d} \mu (w).$$

By setting $z=0$ in the above integral we get $\nabla L_1(0)=0,$ which implies that $$\int_{A} p_a(w)\mathrm{d}\mu(w)=0.$$

This completes the proof.
\end{proof}

\begin{definition}
We say that the point $a$ from the above Theorem is barycenter of the set $A$ w.r. to the measure $\mu$.
\end{definition}

\subsection{Holomorphic barycenter in the Bergman ball}

\begin{definition}
Barycenter of $K \subseteq \mathbb{B}_m$ w.r. to the hyperbolic measure $\Lambda(z)$ defined by (\ref{hyp_meas}) is named {\it holomorphic barycenter} of the set $K$.
\end{definition}

From Theorem \ref{mainth1}(2) it follows that the holomorphic barycenter of the set $A$ is minimum of the function (\ref{marte}) where the measure $\mu(w)$ is replaced by the hyperbolic measure $\Lambda(w)$.

By choosing the counting measure (\ref{count}) in Theorem \ref{mainth2} we infer the result about holomorphic barycenter of a finite set of points.

\begin{corollary}\label{t2}
Assume that $z_1,\dots, z_N$ are points in the unit ball $\mathbb{B}_m\subset \mathbb{C}^m$.
\begin{enumerate}

\item
 There exists a unique (up to a linear unitary transformation) automorphism  $p(z)$ of the unit ball onto itself, such that $$\sum_{k=1}^N p(z_k)=0.$$

\item
Decompose automorphism $p$ as $p=U\circ p_c$ for a certain $c \in \mathbb{B}_m$ and a linear unitary transformation $U$ of the unit ball.
Then the point $c$ is the unique minimum of the function
$$
L_N(z) = - \sum \limits_{i=1}^N \log \frac{(1-|z|^2)(1-|w_i|^2)}{|1- \langle z,w_i\rangle|^2}
$$
\end{enumerate}
\end{corollary}

\begin{definition}
The point $c$ from Corollary \ref{t2}(2) is said to be holomorphic barycenter of the set $\{ z_1,\dots,z_N\}$.
\end{definition}

\begin{theorem}
Holomorphic barycenter is holomorphically invariant. In other words, if $c=c(K)$ is holomorphic barycenter of $K$, then $p(c)$ is holomorphic barycenter of $p(K)$ for any automorphism $p$ of the unit ball.
\end{theorem}

The proof of holomorphic invariance in this case is similar to the conformal case, so we skip it.

\begin{remark} \label{Rem_conf_hol}
\begin{enumerate}

\item
Results and notions from sections \ref{balls} and \ref{Bergman} are equivalent for the dimension $2m=n=2$. More precisely, results regarding both hyperbolic and Bergman balls reduce to those from Section \ref{Intro} for Poincar\' e disk when $2m=n=2$.

\item
From the previous point it follows that conformal and holomorphic barycenters coincide for $n=2$. However, they are different in the case of complex dimension greater than one. Namely, in $\mathbb{C}^m\cong \mathbb{R}^n$, $m=2n>1$, M\"obius self-mappings of the unit ball and holomorphic automorphisms are different mappings and their corresponding metrics $d_h$ and $d_B$ are not equivalent.

\item
In balls of real odd dimension, i.e. in $\mathbb{B}^3, \mathbb{B}^5,\dots$ the meaningful notion of holomorphic barycenter does not exists. In such balls one can talk about conformal barycenters only.

\item
 If $D\subset \mathbb{B}$ is symmetric with respect to $z=0$, then its barycenter (both conformal and holomorphic) is equal to zero. Namely, in that case there is a linear isometry $L$ and a partition $\{D_1,D_2\}$ of $D$,  with $|D_1|=|D_2|$ (of equal measure)  so that $\Omega_1=L(D_1)\subset R_+^n$ and $\Omega_2=L(D_2)\subset R_{-}^n$. Then $\int_{\Omega_1}x \mathrm{d}\lambda(x)= -\int_{\Omega_2}x \mathrm{d}\lambda(x)$, and so $$\int_{D}Lx \mathrm{d}\lambda(x)=\int_{\Omega_1+\Omega_2} y \mathrm{d}\lambda(y)=0$$
where $\lambda(\cdot)$ is the Lebesgue measure. In the same way we prove a similar statement for hyperbolic measure.

\end{enumerate}
\end{remark}


\section{Examples}

\begin{example}
Consider an interior of the ellipse $D = \{x + i y : 4x^2 + 9y^2 < 1\} \subset \mathbb{B}^2$ with semi-axes $1/2$ and $1/3$. Then both Lebesgue and conformal barycenter are equal to zero (see Remark~\ref{Rem_conf_hol}(4)).

Let $D_1 = h_{1/2}(D)$, where $h_{1/2}(z) = \frac{1/2 - z}{1 - z/2}$ (see Figure 1). By our theorem, the conformal barycenter of $D_1$ with respect to the hyperbolic measure is $1/2$. However, $1/2$ is not the conformal barycenter of $D_1$ with respect to the Lebesgue measure. To verify this, by calculation, we have
\[\begin{split}
\int_{D_1} h_{1/2}(z)d\lambda(z) &= \int_{D}w \frac{(1-1/4)^2}{|1-w/2|^4}d\lambda(w)\\&=\frac{3 \pi  \left(4426 \sqrt{139}-281475 \tanh^{-1}\left[\frac{2}{\sqrt{139}}\right]\right)}{111200}\approx 0.336214\neq 0.\end{split}\]

Since the last integral is not zero, $1/2$ is not the barycenter of $D_1$ with respect to the Lebesgue measure. Numerical methods show that the barycenter of $D_1$ is approximately $0.46$, which is less than $0.5 = 1/2$.
\begin{figure}[H]
\centering
\includegraphics{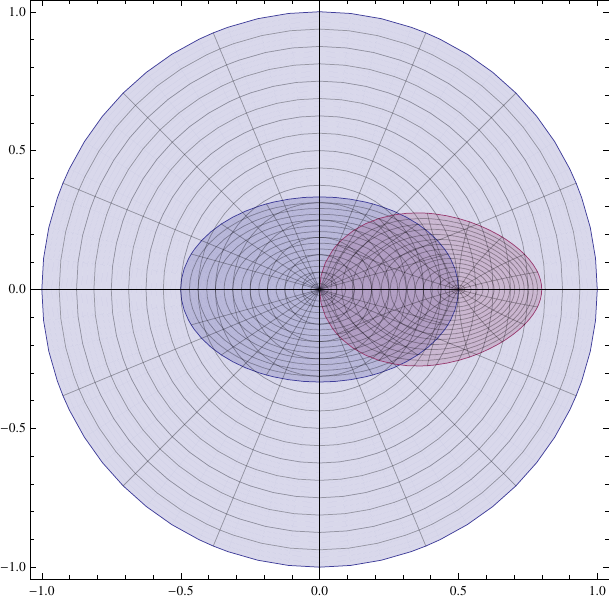}
\caption{Interior of the ellipse $D$ and egg-shape domain $D_1$ in the unit disc}

\end{figure}
\end{example}

The above example demonstrates that the barycenter w.r. to Lebesgue measure is not conformally invariant.



Finally, we consider two simple examples which may provide some additional intuition on barycenters studied throughout the present paper.

\begin{example}
 Let $z_1,z_2\in \mathbb{D}$ and consider
$$g(z) = -\frac{1}{2} \sum_{k=1}^2 \log \frac{(1-|z|^2)(1-|z_k|^2)}{|1-z \overline{z_k}|^2}.$$
Then $g_z=0$ if and only if

$$\frac{\frac{2}{-1+ z \bar z}+\frac{1}{1-z \overline{z_1}}+\frac{1}{1-z \overline{z_2}}}{z}=0$$
whose solutions are
$$z_0=0$$
$$ \hat z= \frac{ \left(1-|z_1z_2|^2-\sqrt{\left(1-|z_1|^2\right) \left(1-|z_2|^2\right)}|1-z_1\overline{z_2}|\right)}{(1-|z_1|^2)\overline{z_2}+(1-|z_2|^2)\overline{z_1}}.$$
and
$$z'=\frac{ \left(1-|z_1z_2|^2+\sqrt{\left(1-|z_1|^2\right) \left(1-|z_2|^2\right)}|1-z_1\overline{z_2}|\right)}{(1-|z_1|^2)\overline{z_2}+(1-|z_2|^2)\overline{z_1}}.$$

Moreover $|\hat z|<1$ and $|z'|>1$. Namely $|\hat zz'|=1$.

Then we easily show that $z_0=0$ is not the minimum of $g$ provided that $z_1+z_2\neq 0$, which implies that the minimum is the second stationary point $\hat z$, because $g(z) \to \infty$ as $|z|\to 1$.
Moreover it can be easily verified that the point $\hat z$ is in the midpoint of the geodesic line between $z_1$ and $z_2$.
\end{example}

\begin{example}
In the same way we prove that $$a=(1+i)\left(\frac{4}{3}+\frac{5 \sqrt{2}}{3 \left(38636+1164 \sqrt{1101}\right)^{1/6}}-\frac{\left(38636+1164 \sqrt{1101}\right)^{1/6}}{3 \sqrt{2}}\right)$$ $$\approx 0.156266 +0.156266 i$$ is the only stationary point of $$g(z) = -\frac{1}{3} \sum_{k=1}^3 \log \frac{(1-|z|^2)(1-|z_k|^2)}{|1-z \overline{z_k}|^2}, \ \ |z| <1 $$ where $z_1=1/2$, $z_2=i/2$ and $z_3=0$.  In this case if $$\varphi_{a}(z) =\frac{a-z}{1-\overline{a}z},$$ then elementary computations yields   that $$\varphi_a(0)+\varphi_a(1/2)+\varphi_a(i/2)=0,$$ which confirms our theorem.
\end{example}

\section*{Ethics declarations}
\subsection*{Conflict of interest}
The author declares that he has not conflict of interest.

\subsection*{Data statement}
Data sharing not applicable to this article as no datasets were generated or analysed during the current study.

\subsection*{Acknowledgments} We are grateful to  Kehe Zhu for his valuable suggestions, in particular regarding the explicit formula for the Bergman distance function. We also thank Dmitry Zaitcev and Luke Edholm for valuable comments.

\end{document}